\documentclass[11pt,reqno,twoside]{amsart}
\usepackage{amsmath}
\usepackage{amssymb}
\usepackage[margin=1.4in]{geometry}
\usepackage{epsfig}
\usepackage{color}
\usepackage{array, boldline, makecell, booktabs}
\usepackage{subcaption}
\usepackage{algorithmicx}
\usepackage[ruled]{algorithm}
\usepackage{algpseudocode}
\usepackage{algc}
\newenvironment{alginc}[1][pseudocode]{\medskip\algsetlanguage{#1}\begin{algorithmic}[0]}{\end{algorithmic}\medskip}

\usepackage{mathtools}

\numberwithin{equation}{section}

\title[
Rationals in the Cantor set]{
The distribution of rational numbers on Cantor's middle thirds set}

\author{Alexander D. Rahm}
\address{University of Luxembourg 
{\tt Alexander.Rahm@uni.lu}}  

\author{Noam Solomon}
\address{Massachusetts Institute of Technology
{\tt noams@mit.edu}}

\author{Tara Trauthwein}
\address{University of Luxembourg 
{\tt tara.trauthwein.003@student.uni.lu}
}

\author{Barak Weiss}
\address{Tel Aviv University
{\tt barakw@tauex.tau.ac.il}}

\newif\ifdraft\drafttrue

\draftfalse

\font\sb = cmbx8 scaled \magstep0

\font\sn = cmssi8 scaled \magstep0

\long\def\combarak#1{\ifdraft{\sb #1 }\else\ignorespaces\fi}

\newcommand\name[1]{\label{#1}{\ifdraft{\sn [#1]}\else\ignorespaces\fi}}

\newcommand\eq[2]{{\ifdraft{\ \tt [#1]}\else\ignorespaces\fi}\begin{equation}\label{eq:
#1}{#2}\end{equation}}

\newcommand {\equ}[1]     {\eqref{eq: #1}}

\newcommand{\CC}{{\mathcal{C}}}

\newcommand{\R}{{\mathbb{R}}}

\newcommand{\Z}{{\mathbb{Z}}}

\newcommand{\N}{{\mathbb{N}}}

\newcommand{\MLO}{\operatorname{MLO}}

\newcommand{\df}{{\, \stackrel{\mathrm{def}}{=}\, }}

\newcommand{\FF}{{\mathcal{F}}}

\newcommand{\til}{\widetilde}

\newcommand{\vre}{\varepsilon}

\font\sb = cmbx8 scaled \magstep0

\newcommand {\ignore}[1]  {}

\newtheorem{thm}{Theorem}[section]

\newtheorem{prop}[thm]{Proposition}
\newtheorem{proposition}[thm]{Proposition}

\newtheorem{remark}[thm]{Remark}

\newtheorem{conj}{Conjecture}

\date{\today}
\begin{document}

\ignore{

\begin{abstract}

\end{abstract}

}

\begin{abstract}
We give a heuristic argument predicting that the number $N^*(T)$ of rationals
$p/q$ on 
Cantor's middle thirds set $\mathcal{C}$ such that 
$\gcd (p,q)=1$ and $q \leq T$, has asymptotic
growth $O(T^{d+\vre})$, for $d = \dim \mathcal{C}$. 
We also describe extensive numerical computations
supporting this heuristic. 
Our heuristic predicts a 
similar asymptotic if $\mathcal{C}$ is replaced with any similar fractal with a
description in terms 
of missing digits in a base expansion. Interest in the growth of $N^*(T)$ is
motivated by a problem of Mahler on intrinsic Diophantine approximation on
$\mathcal{C}$.

\end{abstract}

\maketitle

\section{Introduction}

Let $\CC$ denote Cantor's middle thirds set, i.e. all numbers
represented as  $x =
\sum_{1}^{\infty} a_i 3^{-i}$ with $a_i = a_i(x) \in \{0,2\}$ for all
$i$. Let $N^*(T)$ denote the number of rationals number of the form
$p/q$, with $p$ and $q$ coprime, which belong to $\CC$ and for which
$0< q \leq T$. Motivated by questions in Diophantine approximation,
our goal will be to understand the asymptotic growth rate of $N^*(T)$.

Everything we will say in the sequel will apply with minor
modifications to a more general situation in which
$\CC$ is the set of numbers defined by a restriction in a digital expansion,
i.e. for some integer $b \geq 3$ and some proper subset $\mathcal{F}$
of $\{0, \ldots,
b-1\}$ we will let $\CC$ denote the set of numbers $x = \sum_1^{\infty} {a_i} b^{-i}$
with all $a_i \in \mathcal{F}$. To simplify notation we will stick
throughout to the standard ternary set. When writing a rational as
$p/q$ we always assume that $p$ and $q$ are coprime.  

Fix $c \in (0,1)$, let $I_T$ denote the interval $[(1-c)T,T]$ and let 
\[
\begin{split}
N(T) & \df \# \, \left\{\frac{p}{q} \in \CC : q \in I_T \right\} \\
\til N(T) & \df \# \, \left\{\frac{p}{q} \in \CC  \mathrm{\ purely
  \ periodic}  : q \in I_T \right \} \\ 
\til N^*(T) & \df  \# \, \left\{\frac{p}{q} \in \CC : 0< q \leq
T,\ \frac{p}{q} \mathrm{\ is
  \ purely \ periodic} 
\right\}.
\end{split}
\]
Note that these quantities depend on $c$ but this will be suppressed
from the notation. The notations $A(T)= O(B(T))$ and $A(T) \ll B(T)$ mean that $A(T)/B(T)$
is bounded above by a positive constant, and 
$A(T) \asymp B(T)$ means that
the $A(T) \ll B(T) \ll A(T)$. 
\begin{conj}\name{conj: main}
Let $d$ be the Hausdorff dimension of $\CC$, i.e. $d = \log 2 / \log
3$, and in the general case, $d = \log |\FF|/\log b$. 
For each $\vre>0$ we have 
$\til N(T) = O(T^{d+\vre})$. 
\end{conj}

This conjecture was also made by Broderick, Fishman and Reich in \cite{BFR}. An
upper bound $N(T)=O(T^{2d})$ was obtained by Schleischitz in
\cite[Thm. 4.1]{Schlei}. 
%
%
Our heuristic actually predicts a more precise upper bound
for $\til N(T)$, see Remark \ref{remark: better}. The exponent $d$ is optimal in
view of Proposition \ref{prop: lower bound}. 

Since numbers in $\CC$ are explicitly given in terms of their base
3 expansion, it is possible to count their number as a function of 
the complexity of their base 3 expansions. But this says nothing about
the denominator $q$ in reduced form; 
it may happen that a rational with a complicated 
base 3 expansion corresponds to a reduced fraction $p/q$ with $q$
small. The basic heuristic principle behind Conjecture \ref{conj: main},
is that the two events of having a small denominator relative to
the complexity of the base 3 expansion, and of belonging to $\CC$, are
probabilistically independent. We will make this heuristic more
precise below.

Some computational evidence for Conjecture \ref{conj: main} is given
in \cite{BFR}. Our goal in this paper is to present more evidence supporting
it. We will prove that the conjectured asymptotics are
lower bounds for $N^*(T)$ and $\til N^*(T)$; we will describe extensive computations
consistent with this conjecture; and we will discuss the heuristic
motivating Conjecture \ref{conj: main}, exhibiting some numerical results
which lend some support to this heuristic.

\subsection*{Organization of the paper}
In \S\ref{section: history} we discuss some problems in Diophantine
approximation which led us to this problem, and derive a Diophantine
consequence from Conjecture \ref{conj: main}. In \S\ref{section:
  notation} we discuss basic properties of base 3 expansions, which
yield lower bounds on $N^*(T)$ and $\til N^*(T)$. We also explain that
the main quantity of interest is $\til N(T)$. In \S\ref{section:
  heuristic} we introduce a simple 
probabilistic model and use it to predict $\til N(T)$. Some
oversimplifications in the probabilistic models lead to incorrect
predictions, and we modify the model slightly in \S\ref{section:
  corrections} to remedy this, at the same time showing that the
revised model makes the same predictions for the growth of the
expectation of $\til
N(T)$. We discuss fluctuations and the relation of expectations to
asymptotic behavior, in \S\ref{section: fluctuations}. 
Our computational
evidence for our conjectures are given throughout the paper. 

\subsection*{Acknowledgements}
A. Rahm would like to thank Gabor Wiese and the University of Luxembourg for funding his research. 
\textit{Noam's funding: To be updated.}
The research of B. Weiss was supported by ISF grant 2095/15 and BSF
grant 2016256.

\section{Motivation and historical background}\name{section: history}
The classical problem in Diophantine approximation may be formulated
as follows. Given a decreasing function $\varphi: \R_+ \to \R_+$ and a
real number $x$, are there infinitely many rationals $p/q$ such that $|x
- p/q|< \varphi(q)$? In case this holds one says that $x$ is {\em
  $\varphi$-approximable}. For some choices of $x$ and $\varphi$,
determining whether $x$ is $\varphi$-approximable is considered hopelessly
difficult (e.g. $\varphi(q) = 10^{-100}/q^2$, with $x =
2^{1/3}$ or $\pi$); a fruitful line of research is to  
fix $\varphi$ and ask about the measure of $\varphi$-approximable
numbers, with
respect to some measure. Some classical results in diophantine
approximation are:
\begin{itemize}
\item[(Dirichlet)] Every $x$ is $1/q^2$-approximable. 
\item[(Khinchin)]  With respect
  to Lebesgue measure, if $\sum q\varphi(q)$ converges then almost
  no $x$ is $\varphi$-approximable, and if $\sum
 q \varphi(q)$ diverges then almost
 every $x$ is $\varphi$-approximable. 
\item[(Jarn\'ik)] The set 
$$\mathrm{BA} \df \left \{ x : \exists c>0 \mathrm{\ s.t. \ } x
  \mathrm{\ is \ not \ } c/q^2\mathrm{-approximable} \right \}$$
has Hausdorff dimension 1, but Lebesge measure zero. 
\end{itemize}

One measure to consider in place of Lebesgue measure
in such statements, is the coin tossing measure (assigning equal
probability 1/2 to the digits 0,2 in base 3 expansion) on Cantor's ternary set
$\CC$. We give a brief list of activity concerning this type of
question. 

In 1984, Mahler \cite{Mahler} asked how well numbers in $\CC$ can be
approximated 
\begin{itemize}
\item[(i)] by rationals in $\R$. 
\item[(ii)] by rationals in $\CC$. 
\end{itemize}
Question (i) can be formalized in various ways, e.g. for which functions
$\varphi$, does $\CC$ contain $\varphi$-approximable numbers? For
which $\varphi$ is almost every number in $\CC$ (with respect to the
natural coin-tossing measure)
$\varphi$-approximable? for which $\varphi$ is the set of numbers in
$\CC$ which are $\varphi$-approximable of the same Hausdorff dimension
as that of $\CC$? There has been a lot of recent activity concerning
these and similar questions, see
\cite{cantor, Fishman, LSV, Bugeaud, SW} and the references therein. 

Question (ii), which is referred to as an {\em intrinsic
  approximation} problem, has not been nearly as well-studied. 
Broderick, Fishman and Reich \cite{BFR} proved an analogue of
Dirichlet's theorem for Cantor sets and other missing digit sets. 
Fishman and Simmons \cite{FS} extended the main  result of \cite{BFR}
to a more general 
class of fractal subsets of $\R$. A major difficulty in intrinsic
approximation problems is that
there is no reasonable understanding of the growth of the function
$N(T)$, $\til N(T)$ as described above; 
bounds on these functions will yield some
progress on Mahler's question (ii). In particular, 
Conjecture \ref{conj: main}
implies (see \cite{BFR} for the derivation): 
\begin{conj}\name{conj: mahler}
For almost every $x \in \CC$, with respect to the coin-tossing
measure, for any $\vre>0$, there are only finitely many rationals $p/q
\in \CC$ such that 
\eq{eq: approx}{
\left |x - \frac{p}{q} \right| < \frac{1}{q^{1+\vre}}.
}
\end{conj}


It was shown in 
\cite{BFR} 
that for each
$x \in \CC$, there are 
infinitely many $p/q \in \CC$ for which $|x - p/q| < q^{-1}(\log
q)^{1/d}$. Thus the exponent in \equ{eq: approx} cannot be
improved.  
\section{Notation, basic observations, and a lower bound}\name{section:
  notation}
The number $x = \sum_1^\infty a_i(x)3^{-i} $ is rational if and only if
the
sequence $(a_i(x))_{i \geq 1}$ is {\em 
  eventually periodic,} 
i.e. there are integers $i_0 = i_0(x)\geq 0$ and $\ell=\ell(x)>0$, called
respectively the {\em length of initial 
  block} and {\em period}, such that
\eq{eq: defining period}{
  a_i(x)= a_{i+\ell}(x), \ \ \mathrm{for \ all \ } i>i_0,}
and \equ{eq: defining period} does not hold for any smaller $i_0$ or $\ell$.
We say that $x$ is {\em purely periodic} if $i_0=0$. It is elementary
to verify the following (see also \cite[Lemma 2.3]{BFR}):

\begin{prop}\name{prop: rationals base 3}
Suppose $x$ is a rational in $\CC$, with $(a_i)$, $i_0$ and $\ell$ as above. Then we
may write $x = P/Q$ where 
$$P = \sum_{j=0}^{i_0} a_j 3^{i_0+\ell-j} - \sum_{j=0}^{i_0} a_j
3^{i_0-j} + \sum_{j=1}^\ell a_{i_0+j} 3^{\ell-j}, \ \ \mathrm{and} \ Q = 3^{i_0}(3^{\ell}-1)$$ 
(this fraction need not be reduced). In
particular:
\begin{itemize}
\item
  If $x$ is a rational in $\CC$ with period $\ell$ and initial block
  of length $i_0$, then 
  there is an integer $N$ such that 
  $3^{i_0}x-N$ is a purely periodic rational in $\CC$ with period $\ell$.  

\item
if $x = p/q$ is purely periodic
where
$\gcd(p,q)=1$, then $q$ is a divisor of $3^{\ell}-1$ and $\ell$ is the order
of $3$ in the multiplicative group $(\Z/q\Z)^{\times}$.  
\end{itemize}
  \end{prop}
As mentioned above, throughout this paper, the notation $x=p/q$ will
mean that $x$ is a 
{\em reduced} rational in $\CC$, i.e. $\gcd(p,q)=1$. The notation $x =
P/Q$ will mean that $x$ is a rational in $\CC$, not necessarily
reduced.  

The following proposition follows from standard calculations and is
left to the reader. 

\begin{prop}\name{basic}
Fix $c, c' \in (0,1)$ and define $\til N(T)$ and $\til N'(T)$ using
$c$ and $c'$ respectively. Fix $\vre>0$.  If $\til N(T) \ll
T^{d+\vre}$ then the same holds for $\til N'(T), N(T), \til N^*(T)$,
and $N^*(T)$. 
\end{prop}

%

\begin{prop}\name{prop: lower bound} There is $c_1>0$ such that for
  all $T>3$ we have $\til N^*(T) \geq T^d/2$ and $N^*(T) \geq c_1 \log
  (T) \, T^d$. 
\end{prop}
\begin{proof}
Let $\ell = \lfloor \log_3 T \rfloor \geq 1$, i.e. $T \in [3^{\ell},
  3^{\ell+1}]$. There are $2^\ell$ 
purely periodic Cantor rationals of the form $P/Q$ with $Q =
3^{\ell}-1$. Bringing them to reduced form, they are of the form $p/q$
with $q \leq T$. In particular 
$$\til N^*(T) \geq 2^{\ell}  = \left(3^{\ell+1} \right)^d /2 \geq
T^d/2.$$

Similarly any rational of the form $P/Q$ where $Q = 3^{i_0}(3^{\ell -
  i_0}-1)$ will contribute to $N^*(T)$. For each such $Q$, there are
$2^{\ell - i_0}$ possibilities for the digits in the periodic part
of $P/Q$, and $2^{i_0}$ for the digits in the initial block. An
exercise involving the inclusion/exclusion principle (which we omit),
implies that the repetition in this counting is negligible, i.e., up
to a constant,
the number of
distinct rationals $P/Q$ written in this form is at least $ \ell \,
2^{\ell}.$ This proves the claim.
\end{proof}
\section{The heuristic}\name{section:
  heuristic}
In this section we justify 
an upper bound of the form $\til N(T) = O(T^{d+\vre})$. 
Our approach is to assign to each reduced rational $p/q$
a probability that it belongs to $\CC$, and bound the expectation of
the random variable $\til N(T)$ with respect to this
probability. 
Let 
$Q= 3^{\ell}-1$ and consider the rationals $P/Q$ in the interval
$[0,1]$. There are $3^\ell$ such rationals, and of these, $2^{\ell}$
  belong to $\CC$. By Proposition \ref{prop: rationals base 3}, they
  are precisely the purely periodic 
  Cantor 
rationals with period dividing $\ell$. 
That is, fixing $Q$, the proportion of rationals $P/Q \in[0,1]$ which belong to $\CC$ is
$\left(\frac{2}{3}\right)^{\ell}$.

Motivated by this we define our probabilistic model. 
By Proposition \ref{prop: rationals base 3}, $p/q \in \CC$ is purely periodic if
and only if $3$ does not divide $q$.  
For each rational $p/q \in [0,1]$, with $q$ not divisible
by $3$, our model stipulates:

\medskip

(*) {\em 
 The probability that $p/q \in \CC$
is $\left(\frac23 \right)^{\ell}$, where $\ell = \ell(q)$ is the smallest number
for which $Q = 3^\ell-1$ is divisible by $q$; the events $p/q \in \CC$
are completely independent.}

\medskip 

Note that $\ell$ is the order of $3$ in the multiplicative group 
$C_q \df (\Z/q\Z)^\times$. 
 Let $\phi(q) = \# \, C_q$ be the Euler number of $q$. 
We may take representatives of elements of $C_q$
to be the integers $p $ between $0$ and $q-1$ coprime to
$q$, so we find that 
the expected number of $p/q$ in $\CC$ with fixed denominator $q$
  is $\phi(q)\left(\frac23\right)^{\ell(q)}$. 
Thus:
\eq{eq: first expression}{
\begin{split}
\mathbb{E}\left(\til N(T)\right) & = \sum_{q \in I_T}\phi(q)\left(\frac23\right)^{\ell(q)} \\
& \leq \sum_{q \in I_T} T \left(\frac23\right)^{\ell(q)} \\
& = T \, \sum_{\ell \geq \log_3 T + c'} 
\# \, L(\ell, T) 
\left(\frac23\right)^{\ell},
\end{split} 
}
where 
$$L(\ell, T) = \left\{q \in I_T : \ell(q)=\ell \right\}
    \ \ \mathrm{ and } \ \
c' = \log_3 (1-c).$$
We now need to bound the terms $\# \, L(\ell, T)$. 
First we choose $\lambda = \frac{2-d}{1-d}$. For $\ell \geq \lambda\log_3T $ we
can use the trivial bound $\# \, L(\ell, T) \leq T$, since 
$$
T^2 \sum_{\ell \geq \lambda \log_3 T } 
\left(\frac23\right)^{\ell} \asymp
T^{2-\lambda+ \lambda d}= 
 T^d.
$$
So it only remains to show 
\eq{eq: remains}{
 T \, \sum_{\ell = \log_3 T+c'}^{\lambda \log_3 T } 
\# \, L(\ell, T) 
\left(\frac23\right)^{\ell}  = O(T^{d+\vre}).
}
For $\ell \in [\log_3T + c',\lambda \log_3T ]$, we use
the obvious inequality $\# \, L(\ell, T) \leq \tau\left(3^{\ell}-1\right)$,
where $\tau(n)$ denotes the number of divisors of $n$. 
It is well-known 
that 
\eq{eq: tau bound}{
\tau(n) \leq
2^{(1+o(1))\log n/\log\log n}. 
}
In our situation we have $3^{\ell}-1 \leq T^{\lambda}$, so  
$$\tau\left(3^{\ell}-1\right) \leq 2^{2 \lambda \log T /\log \log T}
= T^{2\lambda/\log \log T},
$$
implying
\[
\begin{split}
 T \, \sum_{\ell = \log_3 T+c'}^{\lambda \log_3 T } 
\# \, L(\ell, T) 
\left(\frac23\right)^{\ell}  & \leq T \lambda \log_3 T  T^{2\lambda/\log
  \log T}\left(\frac23\right)^{\log_3T+c'} \\
& \ll \log T \, T^{d+2\lambda/\log\log T}.  
\end{split}
\]
from which \equ{eq: remains} follows.

\ignore{
, and this will require
another heuristic argument. Note that understanding 
precisely the behavior of the function $\ell(q)$ is a 
notoriously difficult problem (cf. Artin's
conjecture).

However for our 
purposes, a simple heuristic bound suffices. Clearly $q$ cannot divide
$3^{\ell}-1$ if $\ell < \log_3q$. Otherwise, we expect $3^{\ell}-1 \mod q$ to be 
uniformly distributed in the additive group $\Z/q\Z$, so the
probability that $q$ divides $
3^{\ell}-1$ is $1/q$. We find that the expected number of $q \in I_T$
for which $q | 3^{\ell}-1$ is $\frac{cT}{q} \leq \frac{c}{1-c}$.

Reworking this into \equ{eq: first expression} gives 
\eq{eq: gives}{
\begin{split}
 & \sum_{q \in I_T}\phi(q)\left(\frac23\right)^{\ell(q)}  \leq  \sum_{q \in
I_T} T \left(\frac23\right)^{\ell(q)} \\
& = 
T \, \sum_{\ell \geq \log_3 T + c_1} \left| \left\{q \in I_T : \ell(q)=\ell \right\}
   \right| \left(\frac23\right)^{\ell} \\  
& \leq T \, \sum_{\ell \geq \log_3 T + c_1} \left| \left\{ q \in I_T : q | 3^{\ell}-1 \right\}
   \right| \left(\frac23\right)^{\ell} \\  
& \leq  T \, \sum_{\ell \geq \log_3 T + c_1} c_2
   \left(\frac23\right)^{\ell} = 3 c_2 T \left(\frac23\right)^{\lfloor
     \log_3 T + c_1 \rfloor} \leq c_3 T^d, 
\end{split}
}
where 
$$c_1 = \log_3 1-c, \ c_2 = \frac{c}{1-c}, \ c_3 = 3c_2\left(\frac23 \right)^{c_1-1}.$$ 

}
\begin{remark}\name{remark: better}
1. In \equ{eq: first expression} we used the inequality $\phi(q) \leq
q \leq T$. But in fact it is well-known that on average $\phi(q) \asymp q$, 
so we actually expect
\eq{eq: actually}{\til N(T) \asymp  T \, \sum_{\ell = \log_3 T + c'}^{\lambda \log_3T}
\# \, L(\ell, T) 
\left(\frac23\right)^{\ell}.}

2. 
Our arguments show that the right hand side of \equ{eq: actually}
behaves like $O(\log T \, 
\, T^{d+2\lambda/\log\log T})$. In
estimating the cardinality of $L(\ell,T)$ we used the bound \equ{eq: tau bound} which is
optimal for a general $n$. However it may be that for numbers of the
form $n=3^{\ell}-1$ a better bound exists, see \cite{Erdos} for
related results. If so then our heuristic would predict a better bound
for $\til N(T)$.

\end{remark}

\section{A revised model}
\name{section: corrections}
The heuristic above relied on the basic statement (*). However this
assumption leads to some clearly
incorrect predictions, namely: 
\begin{itemize}
\item[(i)] {\em (Primitive words)}
In deriving (*) we calculated the frequency of purely periodic rationals
with period {\em dividing} $\ell$, belonging to $\CC$. It would have
been more precise to count the purely periodic rationals with period 
{\em exactly} $\ell$, belonging to $\CC$. By Proposition \ref{prop:
  rationals base 3}, rationals with 
period exactly $\ell$ correspond to {\em primitive} words $w$ in the alphabet $\{0, 1,
2\}$ of length $\ell$, i.e. those $w$ for which there is no proper
divisor $k$ of $\ell$ such that $w$ is a concatenation of a identical
words of length $k$. A standard 
application of the inclusion/exclusion principle
gives that the number
of primitive words of length $\ell$ from an alphabet of size $a$ is
\eq{eq: mobius}{m(\ell, a) \df \sum_{d \mid \ell}
  \mu\left(\frac{\ell}{d}\right) a^d,
}
where $\mu$ is the M\"obius
function. 

\item[(ii)]{\em (Multiples of $\ell$)}
Fix $q$ and let
\eq{eq: defn nq}{
N_q = \# \{p: p/q \in \CC\},
}
and let $\ell = \ell(q)$. Since $\CC$ is invariant under multiplication by
$3$ mod 1, whenever $p/q \in \CC$ we also have $p'/q \in \CC$, where
$p' = 3p \mod 1$. This means that the set $\{p: p/q \in \CC\}$
consists of orbits for the action of $3$ on $C_q$, and in particular,
$\ell$ divides $N_q$. 

\item[(iii)]{\em (Divisibility by 2)}
Let $Q=3^{\ell}-1$. Our model predicts that there are $\phi(Q)(2/3)^\ell$ rationals
in $\CC$ with denominator $Q$, coming from $P \in \{0, \ldots, Q-1\}$ such that
$P/Q$ belongs to $\CC$ and $\gcd(P,Q)=1$. However $Q$ is even and if
$P/Q$ is in $\CC$ then so is $P$, since it may be written in base 3 using
the letters 0 and 2 only. That is, the actual number is zero. A
similar observation holds for any $q$, which divides $Q=3^{\ell}-1$
but does not divide
$Q/2$. 

\end{itemize}

One may define a revised model as follows: 
for each $q$, 
let $H$ be the group generated by $3$ in $C_q$. 
By observation (ii), for each coset $X \in C_q / H$, all
number of the form $p/q, p\in X$ simultaneously belong or do not
belong to $\CC$; if they all do, we will write $X \in \CC$. With this
notation, our revised model stipulates that:

\medskip

(**) {\em Suppose $q$ is not divisible by 3 and divides $(3^{\ell}-1)/2$,
where $\ell = \ell(q).$ For each $X \in C_q / H$, the probability that
$X \in \CC$ is 
$\frac{m(\ell,2)}{\bar{m}(\ell,3)}$, where $m(\ell, a)$ is defined by
\equ{eq: mobius} and $\bar{m}(\ell, a)$ is the set of primitive words
of length $\ell$ in the symbols $\{0,1,2\}$ defining even numbers. }

\medskip

Note that our choice of probability takes into account (i) and (iii). 
It is not hard to show that 
$$
\frac{m(\ell,2)/\bar{m}(\ell,3)}{(2/3)^{\ell}} \to_{\ell \to \infty} 2,
$$
and using this, that the arguments given in \S\ref{section: heuristic} also apply to the
second model, yielding the same prediction. That is, model (**) also
implies Conjecture \ref{conj: main}. Moreover, when $\ell$ is prime,
it is easy to check using \equ{eq: mobius} and the definition of
$\bar{m}$ that the difference between $2 \left(\frac23 \right)^{\ell}$ and
$\frac{m(\ell,2)}{\bar{m}(\ell, 3)}$ is negligible. 
Nevertheless, when testing our heuristic, there will be a 
difference between models (*) and (**). 
For sufficiently small values
of $q$ we have computed the actual
values of $N_q$ as defined in \equ{eq: defn nq}, and one may compare
them to the number  
\eq{eq: def MLO}{
\MLO(q) \df \mathrm{round}\left (\frac{\phi(q) \cdot m(\ell,2)}{\bar{m}(\ell, 3)} \right).
}
See 
Figures~\ref{c05_total} and~\ref{c05_ratio}.

\begin{figure}[!ht]
\centering
\includegraphics[width=\textwidth]{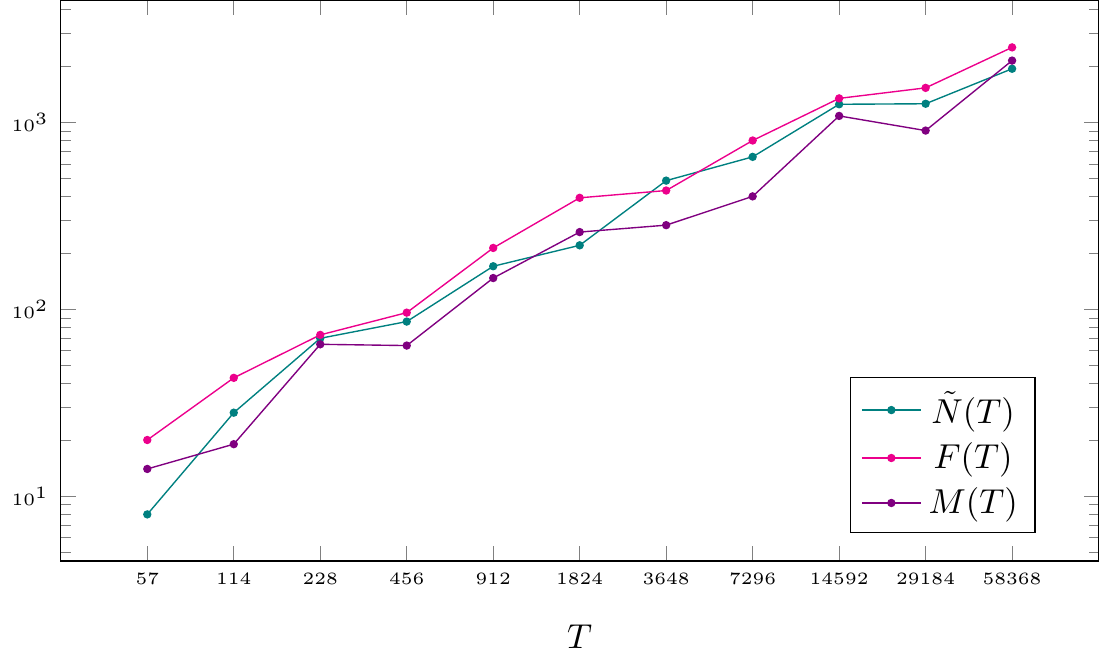}
\caption{The summed number of purely periodic Cantor rationals $\tilde{N}(T)$, 
its approximation $F(T) := \sum_{\stackrel{q \in I_T}{3 \nmid q}} \text{round}\left(\left(\frac{2}{3}\right)^{\ell(q)} \cdot 2 \cdot \phi(q)\right)$ from model (*),
and its approximation $\displaystyle{M(T) := \sum_{\stackrel{q \in I_T}{3 \nmid q|\frac{3^{\ell(q)}-1}{2}}}
\MLO(q)}$ from model (**), where $I_T := [(1-c)T,T]$ for
$c=\frac{1}{2}$. 
More data points shown in Figure \ref{various_c}.} 
\label{c05_total}
\end{figure}

\begin{figure}[!ht]
\centering
		\includegraphics[width=\textwidth]{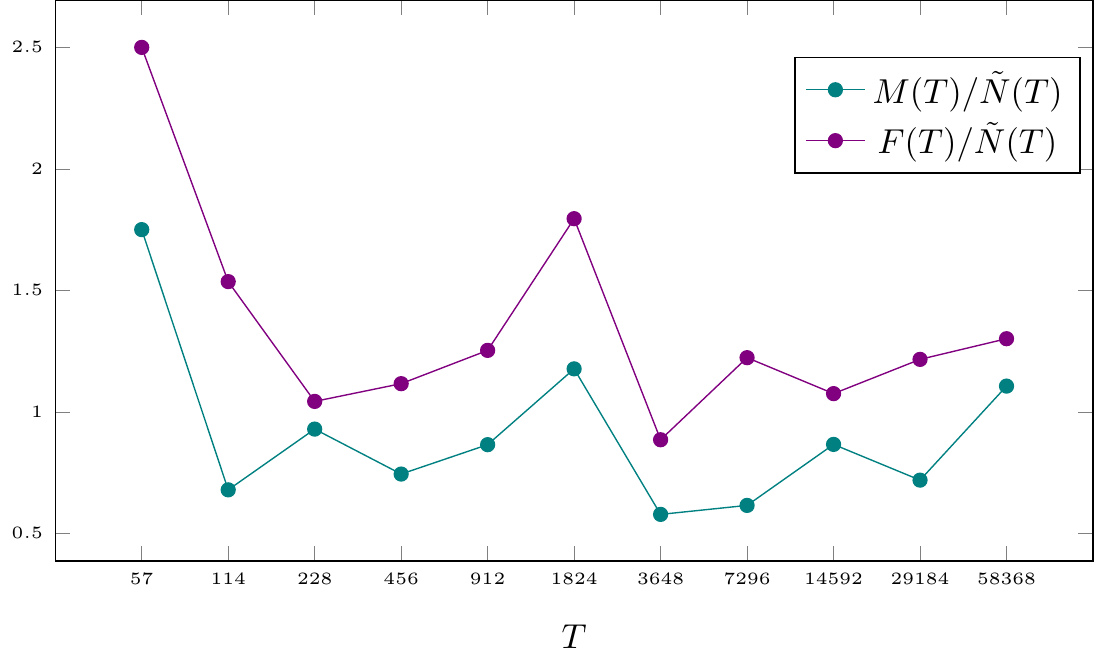}
\caption{Ratios $\frac{M(T)}{\tilde{N}(T)}$ and $\frac{F(T)}{\tilde{N}(T)}$ for
  $c=\frac{1}{2}$. 
  Our heuristic predicts that this graph
  tends to 1 at infinity.} 
		\label{c05_ratio}
\end{figure}



The notation $\mathrm{round}(x)$ stand for the closest integer to $x$,
and the letters MLO stand for {\em most likely
  outcome}, since there is no other number more likely to occur as the
value of $N_q$, under probabilistic model (**). 

Using inclusion/exclusion and M\"obius
inversion, one can show (for more details see \cite{Trauthwein}) that the number
of even (as numbers in base 3) 
primitive words of length $\ell$ with symbols in the alphabet
$\{0,1,...,a-1\}$ is
$$\sum_{d | \ell, \, \frac{\ell}{d} \text{ even}}
\mu\left(\frac{\ell}{d}\right) a^d + \sum_{d | \ell, \frac{\ell}{d} \text{
    odd}}  \mu\left(\frac{\ell}{d}\right)  \left \lceil \frac{a^d}{2} \right
\rceil 
.$$

As a consequence one obtains a simple formula for
$\bar{m}(\ell, 3)$. 
This allows us to compute MLO$(q)$ and hence to plot
Figures~\ref{c05_total} and~\ref{c05_ratio}. As can be seen in the
Figures,  
within the range of our database of Cantor rationals,
{\em both models (*) and (**) give good approximations for the number of
purely periodic Cantor rationals.} The fit is not perfect though, and
the plots reveal other interesting features. We try to explain some of
these below.


%


\section{Remarks on fluctuations, Bourgain's
  theorem, and symmetries}\name 
{section: fluctuations}

\begin{table}
$$ \begin{array}{|r|l|}
\hline
 q_n & \ell(q_n)/\log_3q_n \\
 \hline
q_0=  3 & 1.0\\
q_1= 30 & 1.292030029884618\\
q_2= 84 & 1.4876881693076203\\
q_3=146 & 2.6453427135663814\\
q_4=386 & 2.951356044207975\\ \hline
 \end{array}$$
\caption{Denominators $q_n$ such that for all $q<q_{n+1}$
admitting Cantor rationals of denominator $q$, 
$\ell(q)/\log_3q \leq \ell(q_{n})/\log_3q_{n}$.
For all $q<3^{10}$
admitting Cantor rationals of denominator $q$, we have
$\ell(q)/\log_3q \leq \ell(q_{4})/\log_3q_{4}$.
}
  \label{Bourgain-table}
\end{table}
 \begin{figure}
 \centering
    \begin{subfigure}[b]{0.3\textwidth}
        \includegraphics[width=\textwidth]{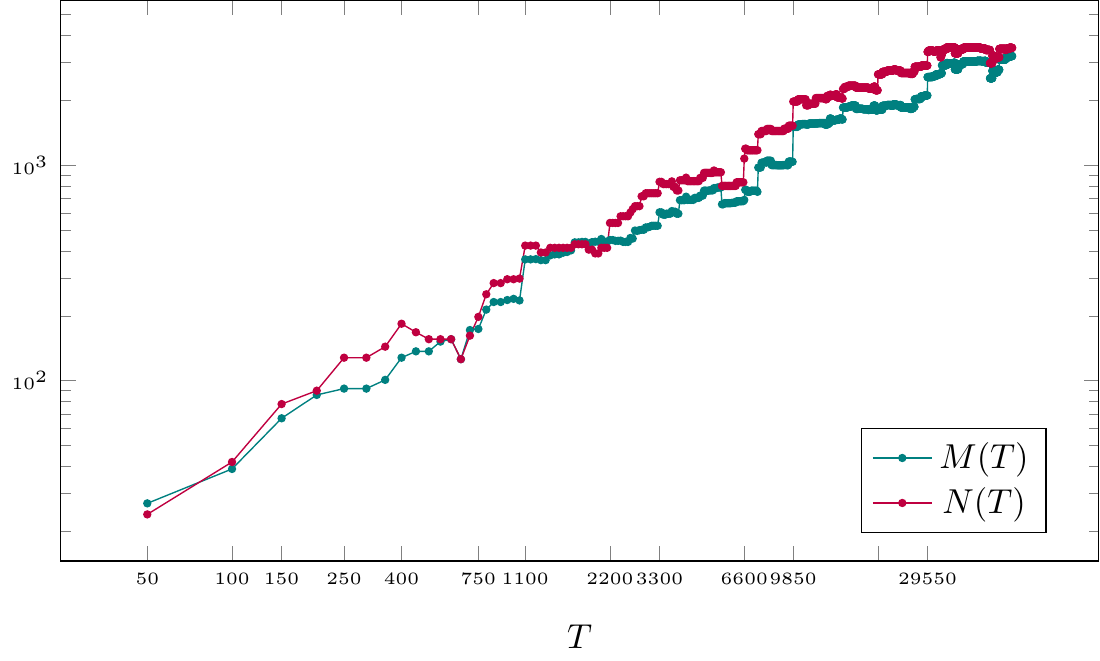}
        \caption{$c =0.8$}
        \label{0dot2}
    \end{subfigure}
    \begin{subfigure}[b]{0.3\textwidth}
        \includegraphics[width=\textwidth]{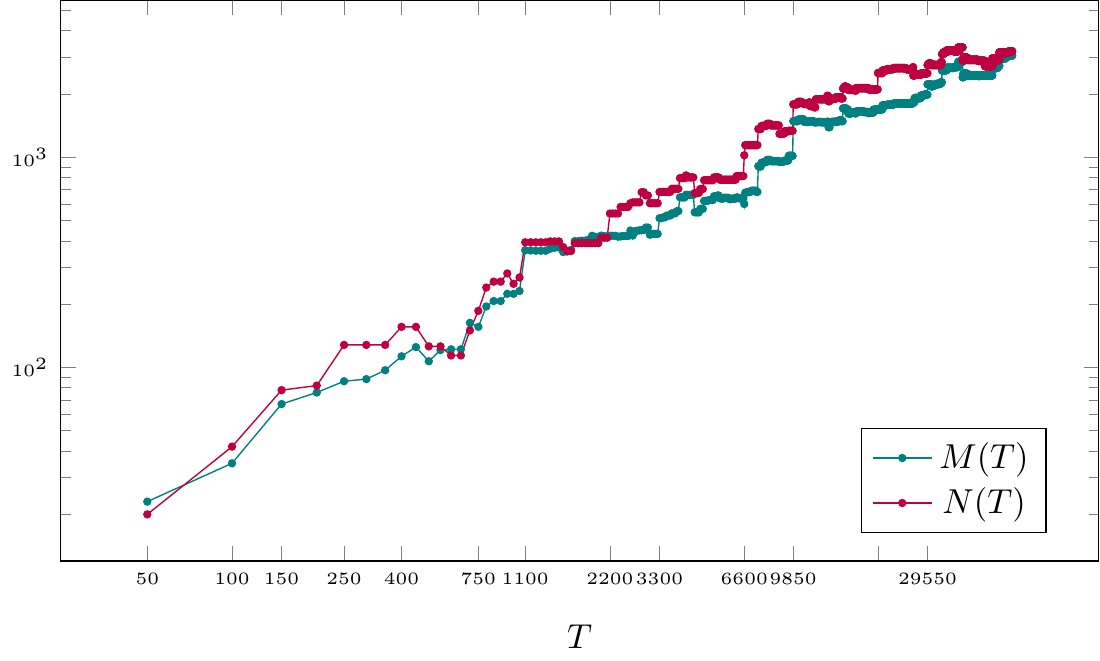}
        \caption{$c =0.75$}
        \label{0dot25}
    \end{subfigure}
    \begin{subfigure}[b]{0.3\textwidth}
        \includegraphics[width=\textwidth]{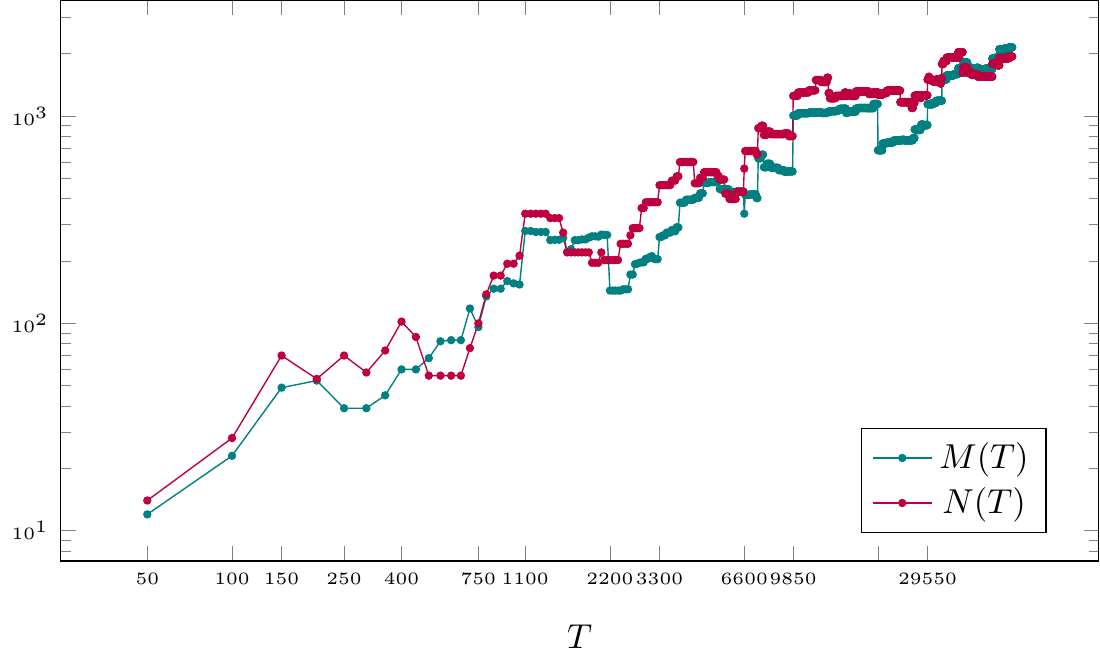}
        \caption{$c =0.5$}
        \label{0dot5}
    \end{subfigure}
    
     \begin{subfigure}[b]{0.45\textwidth}
        \includegraphics[width=\textwidth]{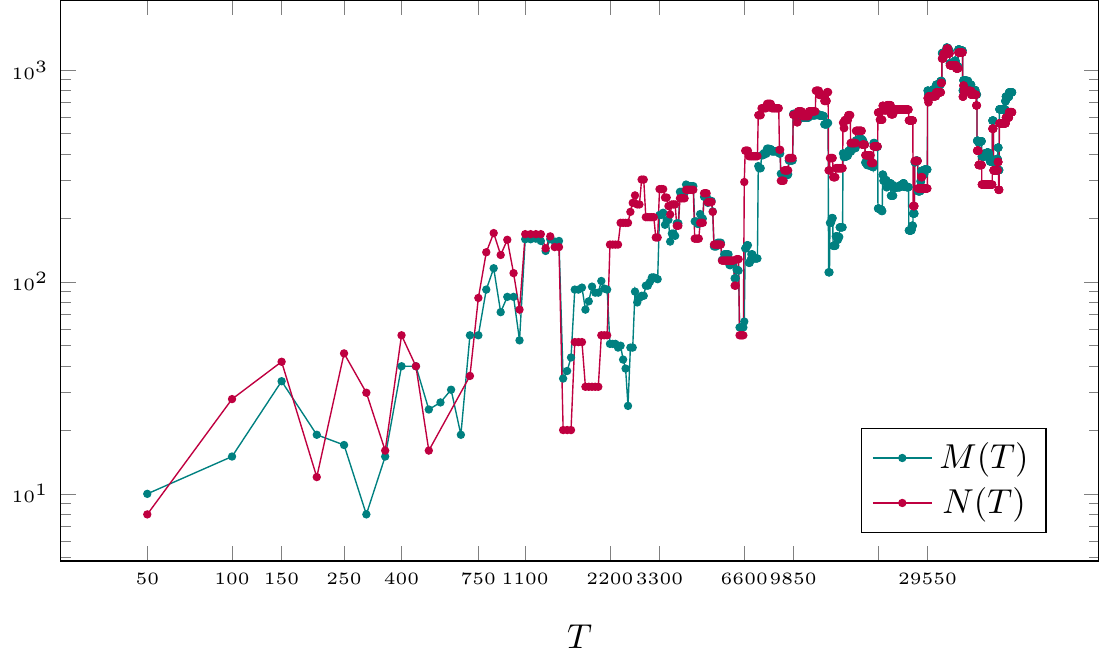}
        \caption{$c =0.25$}
        \label{c=0dot25}
    \end{subfigure}
    \begin{subfigure}[b]{0.45\textwidth}
        \includegraphics[width=\textwidth]{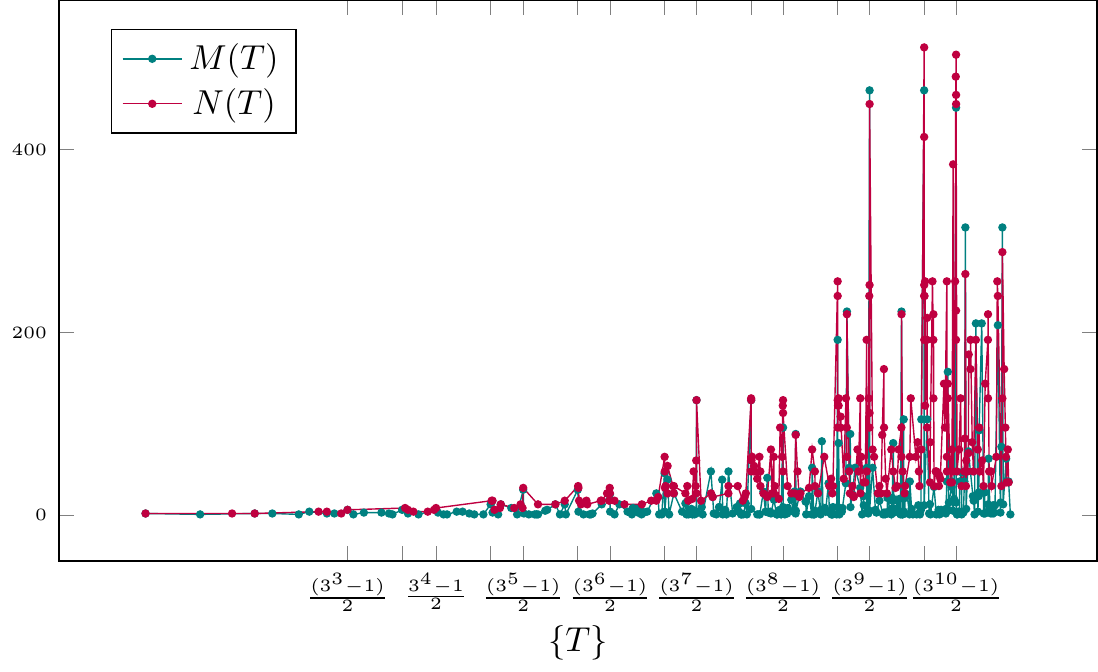}
        \caption{$c =0$}
        \label{c=0}
    \end{subfigure}   
   \caption{For different values of $c$ (which determine the intervals
     $I_T := [(1-c)T,T]$), we plot the summed number of purely periodic Cantor rationals $\tilde{N}(T)$
and its approximation $M(T) $ 
from model (**). As predicted in \S \ref{subsec:
  fluctuations}, there are  more fluctuations for smaller $c$.}\label{various_c}
   \end{figure}
\begin{table}
$$
 \begin{array}{|l|c|c|c|c|c|
}
  \hline
r & q = 3^r+1           & N_q 
& {\rm MLO}(q) & \frac{N_q}{{\rm MLO}(q)} & \left(\frac{2}{3} \right)^r \frac{N_q}{{\rm MLO}(q)} 
   \\ 
\hline &&&&&
   \\
   4 &82  & 16  & 3   & 5.333 & 1.053 \\
   5 &244 & 30 &4    &7.5 & 0.988\\
   6 &730  &48 &4 &12 & 1.053 \\
   7 & 2188 &126 & 7&18 &1.053 \\
   8 & 6562& 240 & 9  & 26.667 & 1.04 \\
   9 &19684 & 414 &11  & 37.636 & 0.979\\
   10 & 59050& 820 & 14& 58.571 & 1.016 \\
   11 & 177148& 2024& 23 & 88 & 1.017\\
    12&   531442 & 4008 & 31 & 129.29 &  0.996 \\
 13 &  1594324  & 8190 & 42  & 195 &  1.002\\ 

\hline
\end{array}
$$
\caption{The numbers $q = 3^r+1, \  r = 4, \ldots, 13$ where our
  heuristic gives poor 
  predictions. When revising the
  prediction by a factor of 
  $(3/2)^r$, which is the factor taking into account a symmetry
  $\omega \mapsto \omega \bar{\omega}$, we obtain a much better prediction. }
\label{symmetries1}
\end{table}

\begin{table}
$$
 \begin{array}{|l|c|c|c|c|
}
  \hline
q           &\ell(q)& N_q 
& {\rm MLO}(q) & \frac{N_q}{{\rm MLO}(q)} 
   \\ 
\hline &&&&
   \\
 12962 & 24 & 72 & 1 & 72 \\
     14965 & 24 & 48 & 1 & 48 \\
    29848 & 24 & 48&1 &  48 \\
    84253 & 24 & 96 & 9 & 10.391 \\ 
  129620 & 24 & 48 & 6 & 8 \\
   181468 & 24 &96  & 9 & 10.391\\
239440 & 24 & 96 & 11 &  8.727 \\ 
259240 &24  & 48& 12 &  4 \\ 
 298480 & 24 & 96 & 11 & 8.727 \\ 
   531442 & 24 & 4008 & 31 & 129.29\\
   589771 & 24& 336 & 55 & 6.109 \\ 
4731130 & 24 & 960 & 222 & 4.324 \\
   21257680 & 24 & 4176 & 985 & 4.24\\
\hline
\end{array}
$$
\caption{All values of $q$ with $\ell(q)=24$ for which our heuristic
  makes a prediction which is incorrect by a factor of 4 or
  more. Note that in all of these examples, $N_q>
  \mathrm{MLO}(q)$. At least three, and probably all,  of the entries
  in the table are related to 
  the 
symmetries discussed in \S \ref{subsec: symmetries}.}\label{symmetries} 

\end{table}

\begin{table}
  $$
   \begin{array}{|l|c|c|c|c|
     }
     \hline
q           &\ell(q)& N_q 
& {\rm MLO}(q) & \frac{N_q}{{\rm MLO}(q)} 
   \\ 
     \hline &&&&
     \\ 23 & 11  &  0 & 0 & - \\
     47 & 23 &  0 & 0 & -
                     \\
683           & 31  &   0 
   &       0      &          -          
\\ 
1597          & 19  &   0 
   &       1      &          0          
     \\ 
1871          & 17  &   0 
 &       4      &          0      
 \\ 
3851          & 11  &  88 
 &      89      &                      0.989   
\\ 28537         & 29  &   0 
 &       0      &          -             
\\ 34511         & 17  &  68 
 &      70      &               0.971      
\\ 102673        & 31  &   0 
 &       1      &          0              
\\
363889        & 19  & 304 
  &     328      &    0.927                   
  \\ 
59\cdot 28537 & 29  &   0 
 &      26      &          0              
 \\ 
4404047       & 31  &  62 
 &     31       &     2                    
\\ 20381027 & 29 & 232 
& 319 & 0.727 
\\ 1001523179 & 23 & 178480 & 178481 & 0.999994 \\
 \hline
\end{array}
$$
\caption{Some numbers $q< \frac{3^{\ell(q)}-1}{2}$ for which
  $\ell(q) $ is 
  a prime, 
  including all such $q$ with $11 \leq \ell(q) \leq 23$.
  In this case, symmetries are impossible and our heuristic works well
for each individual $q$.}
 \label{validatingExperiment}
 \end{table}

\subsection{Deviations from the mean}
An obvious objection to the line of reasoning presented
above, is that our prediction for $\til N(T)$ is
based on bounds on its {\em expectation.} That is, we have shown that
our heuristic implies
$\mathbb{E}(\til N(T))=O(T^{d+\vre})$, but in order to justify $\til
N(T)=O( T^{d+\vre})$ one needs additional arguments, which we now
briefly indicate. 

If for some $\vre>0$ there is an unbounded sequence of $T$ for which
$\til N(T) \geq T^{d+\vre}$, then (possibly modifying the constants 
$\vre$ and $c$) we can take this to be a subsequence of the numbers in the form $T_k =
(1+c)^k$. For each $k$ we let $X_k$ denote the random
variable, in model (*), counting the number of $p/q \in \CC$ with $q
\in I_{T_k}$. We will show that the probability
that $X_k$ exceeds $T_k^{d+\vre}$ is $O(T_k^{-\vre})$, and hence
is summable; from this it 
follows by Borel-Cantelli that the probability that for infinitely
many $k$ we have $X_k \geq T_k^{d+\vre}$ is zero.

We continue to denote by $c', \lambda$ the constants as
in \S \ref{section: heuristic}, and write $T=T_k$ to simplify
notation. Let 
$X_k^{(1)}$ (respectively, $X_k^{(2)}$) be the number of $p/q$ contributing to $X_k$ with $\ell(q) > \lambda
\log_3T$ (respectively, $\log_3T + c' \leq  \ell(q) \leq
\lambda \log_3q$). 
Let $\ell_0 = \lambda \log_3T$, which is a lower bound for $\ell(q)$
when $p/q$ contributes to $X_k^{(1)}$. Since there are fewer than
$T^2$ rationals $p/q$ with $q \in I_T$, the 
probability that $X_k^{(1)} \geq T^{d+\vre}$ is smaller than the
probability that a binomial random variable with probability
$$p=\left(\frac23\right)^{\ell_0} = T^{(d-1)\lambda}$$ 
and $T^2$ trials we will have $T^{d+\vre}$
successes. By the Markov inequality, 
this probability is bounded above
by $T^{2+(d-1)\lambda-d - \vre} = T^{-\vre}.$ 
The proof for $X_k^{(2)}$ is similar, again using the Markov
inequality and the bounds used in the proof of
\equ{eq: remains}. 
\ignore{Namely, the probability that $X_k^{(2)}$ is greater
than $T^{d +\vre}$ is bounded above by the probability that a binomial
distribution with probability 
$$
p=\left(\frac23 \right)^{\log_3 T + c'} \ll T^{d-1}
$$
and $\lambda \log T T^{1+2\lambda /\log \log T}$ trials will have
$T^{d+\vre}$ successes. Once again this probability is bounded above
by $O(T^{-\vre})$ by Markov's inequality.   
}
\subsection{Large $\ell$ and Bourgain's theorem}
To highlight the sensitivity of $\til N(T)$ to fluctuations, 
consider the expression
$$
\widehat{\ell} (q) = \left \{\begin{matrix} 
\ell(q) & N_q \neq 0
\\ 
0 & \mathrm{otherwise}
\end{matrix} \right.  \ \ \ \ \text{(with $N_q$ as in \equ{eq: defn nq});} 
$$
that is, $\widehat{\ell}(q)$ is the order of $3$ in
$(\Z/q\Z)^{\times}$ when there are rationals with denominator $q$ in
$\CC$, and zero otherwise. 
Clearly the nonzero values of $\widehat{\ell}(q)$ range between
$\log_3 q $ and $q$. If one could prove that
$\widehat{\ell}(q) \ll \log_3q$ one would obtain a simple proof of
Conjecture~\ref{conj: main}. Note that the heuristic behind Artin's
conjecture (see \cite{Moree survey}) predicts that there are infinitely many $q$
for which $\ell(q) \gg q$, so that this may appear at first sight to
be wildly optimistic. However our restriction $N_q \neq 0$ is a
stringent one. In fact, our computations found that
for all $3 \leq q < 3^{10}$, $\widehat{\ell}(q) < 
3\log_3q$ (see Table \ref{Bourgain-table}).


 On the other hand,
by observation (ii), a large value of
$\widehat{\ell}(q)$ would make a large contribution to $\til N(T)$
when $q \in I_T$. For example if there were infinitely many $q$ for
which $\widehat{\ell}(q) > q^{d+\vre}$, then their contribution alone
would yield a
contradiction to Conjecture \ref{conj: main}. However, a difficult result of
Bourgain \cite{Bourgain} implies that for any $\delta>0$, $\widehat{\ell}(q) \ll
q^{\delta}$. Bourgain's theorem is much stronger inasmuch as it
implies that the cosets of the subgroup $H$ equidistribute in the
interval $[0,1]$ when $\ell(q)> q^{\delta}$, while to obtain the upper
bound above, one only needs to know that if $\ell(q)> q^{\delta}$,
then any coset for $H$ contains at least one point in the interval
$(1/3, 2/3)$. It
would be of interest to obtain better upper bounds on
$\widehat{\ell}(q)$ than those implied by Bourgain's theorem.

 \subsection{Additional sources of fluctuations}\label{subsec: fluctuations}
 It is easy to show that \equ{eq: actually} predicts a lower bound $\til N(T)
 \gg T^d$. However we do not expect a precise asymptotic in the form
 $\til N(T) \sim 
 cT^d$, that is, we do not expect the limit of $\til N(T)/T^d$ to
 exist. There are two reasons for fluctuations in this
 expression. First consider the numbers of the form $q =  
 (3^{\ell}-1)/2,$ for which $\ell(q)=\ell$. If $c< 2/3$, depending on
 the choice of $T$, the range 
 $I_T$ may or may not contain one such number. In case it does,
 this contributes a term of order $(2/3)^{\ell} \asymp T^d$ to the sum, which would
 contribute to the main term. Thus we have fluctuations according as the
 window $I_T$ does or does not contain such $q$, or for general
 $c \in (0,1)$, depending on the number of such $q$ in the interval
 $I_T$. See Figure~\ref{various_c}.

   Although these fluctuations would contradict a precise
 asymptotic $\til N(T) \sim cT^d$, they do not preclude the weaker
 statement $\til N(T)
 \asymp T^d$. A potentially more serious source of 
 fluctuations in \equ{eq: actually} is the number $\# \, L(\ell, T)$,
 which could fluctuate considerably due to fluctuations in the numbers
 $\tau(3^{\ell}-1)$. It would be interesting to determine the
 asymptotic behavior of the right hand side of \equ{eq: actually}. 

%
 
 \subsection{Symmetries}\label{subsec: symmetries}
Heuristics (*) and (**) can also be used to make predictions for the
number $N_q$ of Cantor 
rationals {\em with a fixed denominator $q$.} However in this regime,
our computations reveal many values of $q$ for which the heuristic gives
inaccurate predictions. Some of these are shown in Tables
 \ref{symmetries1} and \ref{symmetries}. The numbers in Table
 \ref{symmetries1} are all of the form $3^r+1$, and in Table
 \ref{symmetries} we show all numbers $q$ for which $\ell(q)=24$ and
 the prediction is 
 inaccurate by a factor of 4 or more. We will consider a possible
 explanation for these 
 inaccuracies by introducing a (non-rigorous) notion of `symmetries' in base 3
 expansion. 

The identity $\frac{3^{2r}-1}{2} = \frac{(3^{r}-1) (3^r+1)}{2}$ easily
implies the following (we leave details to the reader): suppose a
purely periodic rational in 
base 3 expansion has repeating block $\omega \in \{0,2\}^{r}$,
where $r$ is the length of $\omega$, and
 $\bar{\omega}$ is the block obtained from $\omega$ by
replacing occurences of $0$ with $2$ and $2$ with $0$. Then the word
$\omega \bar{\omega}$ of 
length $2r$ obtained by 
concatenating $\omega, \bar{\omega}$ defines (via an infinite base 3
expansion $0.\omega \bar{\omega} \omega \bar{\omega} \cdots$) a number
in 
$\mathcal{C}$ whose  denominator divides
$3^{r}+1$. This implies that any $\frac{p}{3^{r}-1} \in \mathcal{C}$
gives rise to some $\frac{p'}{3^{r}+1} \in \mathcal{C}$ (and in fact,
by observation (ii) in \S \ref{section: 
  corrections}, to the $\times 3$-orbit of this word, which
typically contains $2r$ numbers). It can be deduced that
heuristic (**) underestimates
numbers $p'/q'$ with $q'$ dividing $3^{r+1}$, arising in this way, by a factor of
approximately $(3/2)^r$. The revised heuristic is borne out by Table
\ref{symmetries1}, where the last column corrects heuristic (**)
by this factor, giving a good fit with the data.

The mapping $\omega \mapsto \omega \bar{\omega}$ used above is for us
an example of a 
symmetry in base 3. Here is 
another example. Suppose $\omega, \bar{\omega}
\in \{0,2\}^r$ are as in the previous paragraph, and suppose
$\mathbf{0}$ and $\mathbf{2}$ denote strings of length
$r$ consisting only of the digit $0$ (respectively $2$). Then one may
check, this time using the identity $3^{3r}-1 =
(3^r-1)(3^{2r}+3^r+1)$, that repeating blocks $\omega \bar{\omega}
\mathbf{0}$ and 
$\omega \bar{\omega} \mathbf{2}$ give numbers in $\mathcal{C}$ whose
denominator  
divides $3^{2r}+3^r+1$. For example, taking $r=7$, we have
$q = 3^{14}+3^7+1 = 4785157,$ our heuristic (**) gives
$\mathrm{MLO}(q) = 1771$, and our computer program finds $N_q =
4158$, which is a poor fit. The number of strings of the form $\omega
\bar{\omega}\mathbf{0}$ and $\omega \bar{\omega} \mathbf{2}$, along with
all their cyclic permutations (taking into 
account observation (ii) in \S \ref{section: corrections}) is
2562. Some of these give a subset of the ones already
considered in heuristic (**), so taking this symmetry into account we
should expect
$2562 \cdot \frac{\phi(q)}{q}  = 2365 \leq N_q$.
This 
indeed gives a 
better (albeit still not very precise) prediction. We suspect that there are
more symmetries contributing to the numbers $N_q$ and hope to return
to this issue in future work.
In Table \ref{table: tara
  symmetries} we have tabulated the numbers $q_r$ for $r=2, \ldots,
10$, along with the numbers of strings of the above form multiplied by
$\phi(q)/q$, and compared this prediction with the actual number of
strings of this form which are reduced rationals with denominator
$q_r$.

\begin{table}
$$
 \begin{array}{|l|c|c|c|c|c|c|
}
  \hline
r         &q_r & N_{q_r} & X_r 
& Y_r & Z_r  & Y_r + \mathrm{MLO}(q_r) 
   \\ 
\hline &&&&&& \\
   1 & 13& 6 & 6 & 6 & 6 & 13\\
  2 & 91 & 12 & 18 & 14 & 12 & 27\\
3 & 757 & 54 & 54 &  54 & 54 & 93\\ 
   4 & 6643&120& 156 & 122 & 120 & 202\\
5 & 59293& 450 & 420 & 388 & 390 & 638\\ 
6 & 532171 & 1368 & 1062 & 978 & 1008 &1641\\ 
7 & 4785157& 4158 & 2562 &2365 & 2436 & 4136\\ 
8 & 43053283 & 9744 & 5976 & 4663 & 4560 &8654\\
9 & 387440173 &38988 & 13608 &13450 & 13500 & 26931\\
10 & 3486843451 & 91440 &30450 & 23224 & 23520 & 50961\\
\hline
\end{array}
$$
\caption{The numbers $q_r = 3^{2r}+3^r+1$ with the contribution of the
  symmetries of the form $\omega \mapsto \omega \bar{\omega}
  \mathbf{0}$ and $\omega \mapsto \omega \bar{\omega} \mathbf{2}$. The
  number $X_r$ counts all strings of length $3r$ of the specified
  form, $Y_r = \left \lfloor X_r \cdot
    \frac{\phi(q_r)}{q_r} \right \rfloor$, and $Z_r$ is the actual
  number of Cantor rationals with 
  denominator $q_r$ of this special form.
}
  \label{table: tara symmetries} 
\end{table}

When $\ell = kr$ for $k, r \in \N, k \geq 2$, we can often make a
similar construction of a repeating block of length $\ell$ which is
composed of $k$ sub-blocks of size $r$ (in the preceding two paragraphs
we gave examples with $k=2,3$). The result will be that for the
numbers
 $$ q =
 3^{(k-1)r}+ 3 ^{(k-2)r} + \cdots + 3^r +1,
$$
$N_q$ will be
significantly larger than predicted by our heuristic. The same will be
true for large divisors $q'$ of such $q$. 
Thus if
$\ell$ has many divisors, there will be many values of $q$ for which
our predictions will be poor. In all of them we expect our
heuristic to give a number which is smaller than the correct value,
and we do not expect such very poor predictions to occur when $\ell$ is prime. These
two expectations are borne out in Tables \ref{symmetries} and
\ref{validatingExperiment} below. We 
invite the reader to try to find explanations for the numbers
appearing in Table \ref{symmetries}; note that we have explained the appearance
of 531442 using a symmetry $\omega \mapsto \omega \bar{\omega}$, and
that 589771 and 84253 are large divisors of  $ 3^{16}+3^8+1 $ and can
thus be explained using the symmetries $\omega \mapsto \omega
\bar{\omega}\mathbf{0}, \ \omega \mapsto \omega \bar{\omega} \mathbf{2}$.

\appendix
\section{Computing the Cantor rationals of given denominator}
In this appendix, we give an algorithm to compute the set of rational numbers in the Cantor set 
of given denominator $q$, namely the Cantor rationals of reduced form $\frac{p}{q}$.
It is stated in Algorithm~\ref{Get_q_Cantor_rationals} below, and has
been implemented by the authors in Pari/GP.  
We denote by $\ell(q)$ the order of the element $3$ in the group
of multiplicative units in the ring $\Z/q\Z$ with $q$ elements.

\begin{algorithm} 
\caption{Computation of the Cantor rationals of denominator $q$}
\label{Get_q_Cantor_rationals}
{
\begin{alginc}

\State \bf Input: \rm A natural number $q$.
\State \bf Output: \rm The set of Cantor rationals of reduced form $\frac{p}{q}$. 
\State
\State Carry out the prime decomposition of $q$. 
\State Create a mask $M$ as the set of multiples of the primes in $q$ satisfying that the multiples are strictly smaller than $q$.
\State Denote by $t$ the multiplicity of 3 in the prime decomposition of $q$. 
\State Let $q' := \frac{q}{3^t}$.
\State Compute $\ell(q') := $ order of $3$ in the multiplicative group of the ring $\Z/q'\Z$.
\State Initialize the \emph{passlist} as an empty list.
\For{ $p$ running from $1$ through $q-1$, }

  \If {$p$ is not an element of the mask $M$ or the \emph{passlist},}

     \State  Let $T := \frac{p}{q}(3^{\ell(q')} -1)3^t$.
     \State  Let $A := T \mod (3^{\ell(q')} -1)$.
     \If{ $A \neq 0 \mod (3^{\ell(q')} -1)$,}
       \State Let $a$ be the lift of $A$ to $\{1,\hdots,3^{\ell(q')} -2\}_{\rm ternary}$.
       \If{ the digits of $a$ are in $\{0, 2\}$,}

     	\State Let $s := \left(\frac{T-a}{3^{\ell(q')} -1}\right)_{\rm ternary}$.
        \If{ the digits of $s$ are in $\{0, 2\}$,}

     	  \State \texttt{The fraction $\frac{p}{q}$ is a Cantor rational.}
     	  \State Record it into the set of Cantor rationals of denominator $q$.
	  \State Add 3-power multiples (if $q \neq 0 \mod 3$) of $p$ 
          \State		and their reflections to the \emph{passlist}.
	\Else
     	  \State \texttt{No 3-power multiples of $\frac{p}{q}$ are Cantor rationals.}
	  \State \mbox{Add 3-power multiples of $p$ and their reflections to the mask $M$.}
        \EndIf
       \EndIf
     \Else
	\If{ the digits of $\left(\frac{T}{3^{\ell(q')} -1}\right)_{\rm ternary}$ or $\left(\frac{T}{3^{\ell(q')} -1}-1\right)_{\rm ternary}$are in $\{0, 2\}$,}
     	  \State \texttt{The fraction $\frac{p}{q}$ is a Cantor rational.}
     	  \State Record it into the set of Cantor rationals of denominator $q$.
	  \State Add 3-power multiples$^{{\rm if} \thinspace\thinspace 3\thinspace \nmid \thinspace q}$ of $p$ and their reflections to the \emph{passlist}.
	\Else
     	  \State \texttt{No 3-power multiples of $\frac{p}{q}$ are Cantor rationals.}
	  \State \mbox{Add 3-power multiples of $p$ and their reflections to the mask $M$.}
        \EndIf
     \EndIf 
  \EndIf
\EndFor
\State Output the rationals $\frac{p}{q}$ for $p$ in the \emph{passlist}.
\end{alginc}
}
\end{algorithm}

\begin{proposition}
The set computed by algorithm \emph{\ref{Get_q_Cantor_rationals}} contains all the Cantor rationals of denominator $q$ for its reduced form. This algorithm terminates within finite time.
\end{proposition}
\begin{proof} ${}$

\begin{itemize}
\item The period length of $\frac{p}{q}$ in the ternary system is given by $\ell(q')$.
Hence, the finite sequence $a$ of ternary digits is precisely the periodical sequence in $\frac{p}{q}$.
Furthermore, 
$$ \frac{s(3^{\ell(q')} -1) +a}{(3^{\ell(q')} -1)3^t} = \frac{p}{q}.$$
So, the sequence $s$ is precisely the sequence of ternary digits preceding the periodical part in the ternary expansion of $\frac{p}{q}$.
By the elementary ternary digits property of the Cantor set, algorithm~\emph{\ref{Get_q_Cantor_rationals}} decides if $\frac{p}{q}$ is a Cantor rational. The mask $M$ allows it to check all suitable fractions~$\frac{p}{q}$.
Here, and for establishing the \emph{passlist}, we make use of the
well-known symmetry of the Cantor set: If $x$ is an element of the
Cantor set, then the same holds for $(1-x)$, $\frac{x}{3}$, and ---
provided that it is in the unit interval --- $3x$. 
\item The loop in algorithm~{\ref{Get_q_Cantor_rationals}} consists of
  $(q-1)$ repetitions, which contain a finite number of finite-time
  steps. 
\end{itemize}
\end{proof}

\begin{remark}
\begin{itemize}
 \item 
The mask $M$ can be omitted and a coprimality check for $(p,q)$ inserted, to obtain a simpler algorithm which is mathematically equivalent to algorithm~{\ref{Get_q_Cantor_rationals}}.
The difference lies in the efficiency: In fact, the mask $M$ is a powerful tool to reduce the time needed to carry out the algorithm, minimizing the number of iterations of most expensive steps, which grows fast with $q$.
\item Even more important for the efficiency is the sub-algorithm testing the belonging of the ternary digits to the set $\{0,2\}$, because the numbers to be tested are incredibly great integers. 
\end{itemize}
\end{remark}

\end{document}